\documentclass[12pt]{amsart}
\usepackage{amssymb,amsmath,amscd}
\usepackage{url}
\usepackage[all,cmtip]{xy}
\usepackage{graphicx}
\usepackage{tikz}
\usepackage{pgfplots}
\usepackage{enumerate}
\usetikzlibrary{matrix,arrows,decorations.pathmorphing}

\usepackage[mathlines,pagewise]{lineno}

\usepackage{graphicx}
\def\Circlearrowleft{\ensuremath{%
  \rotatebox[origin=c]{130}{$\circlearrowleft$}}}

\usepackage{ifthen, xcolor}
\newboolean{commentBoolVar}
\setboolean{commentBoolVar}{true} %SET THIS TO TRUE TO SHOW COMMENTS, FALSE TO HIDE THEM

\newcommand{\colourcomment}[3]
{
\ifthenelse{\boolean{commentBoolVar}}{{\color{#2}(#1: #3)}}{}
}

\newtheorem{theorem}{Theorem}[section]
\newtheorem{lemma}[theorem]{Lemma}
\newtheorem{proposition}[theorem]{Proposition}

\newtheorem{example}[theorem]{Example}
\newtheorem*{theorem*}{Theorem}
\newtheorem*{lemma*}{Lemma}
\newtheorem*{proposition*}{Proposition}
\newtheorem*{corollary*}{Corollary}

\newtheorem{remark}[theorem]{Remark}

\newcommand{\Prim}{\mathrm{Prim}} 
\newcommand{\disc}{\mathrm{disc}}

\newcommand{\Spec}{\mathrm{Spec}} 
\newcommand{\Hom}{\mathrm{Hom}}

\newcommand{\coker}{\mathrm{coker}}

\newcommand{\Z}{\mathbb Z}

\newcommand{\be}{\begin{eqnarray}}
\newcommand{\ee}{\end{eqnarray}}

\newcommand{\mmod}{\hbox{\,mod\,}}

\newcommand{\F}{\mathbb{F}}

\newcommand{\D}{\mathbb D}
\newcommand{\caH}{\mathcal H}

\begin{document}

\title{Hopf Orders in $K[C_p^3]$ in Characteristic $p$}
%\author{Alan Koch}
%\address{Department of Mathematics, Agnes Scott College, 141 E. College Ave., Decatur, GA 30030 USA}
%\email{akoch@agnesscott.edu}
%\author{Timothy Kohl}
%\address{Department of Mathematics and Statistics, Boston University, 111 Cummington Mall, Boston, MA 02215 USA}
%\email{tkohl@math.bu.edu} 
%\author{Paul J.~Truman}
%\address{School of Computing and Mathematics, Keele University, Staffordshire, ST5 5BG, UK}
%\email{P.J.Truman@Keele.ac.uk}
\author{Robert Underwood}
\address{Department of Mathematics and Department of Computer Science, Auburn University at Montgomery, Montgomery, AL, 36124 USA}
\email{runderwo@aum.edu}

\date{\today}

\maketitle

%\begin{abstract}  
%\end{abstract}
%\noindent {\it key words:} group ring; Hopf form; cyclic group; Hopf-Galois structure \par
%\noindent {\it MSC:} \par

%\RU{This is an expansion/revision of previous work.  Thanks to Tim for discussions regarding papers \cite{HP86}, \cite{Pa90}.}

\begin{abstract} Let $p$ be prime, let $K$ be a non-archimedean local field of characteristic $p$ and let $C_p^3$ denote the 
elementary abelian group of order $p^3$.  We give a complete classification of Hopf orders in the $K$-Hopf algebra
$(K[C_p^3])^*$ and under a mild condition compute their dual Hopf orders in the group ring $K[C_p^3]$.  
\end{abstract}

\section{Introduction}

Let $K$ be a non-archimedean local field of characteristic $p$.   Equivalently, $K$ is a field of characteristic $p$ that is complete with respect to a discrete valuation 
$\nu: K\rightarrow \Z\cup \{\infty\}$ with uniformizing parameter $\pi$, $\nu(\pi)=1$.   The valuation ring is 
$R=\{x\in K\mid\nu(x)\ge  0\}$ with unique maximal ideal $\mathfrak{p} = \{x\in R\mid\nu(x)\ge 1\}$ and units 
$U(R)=\{x\in R\mid\nu(x)=0\}$.  

Let $B$ be a finite dimensional Hopf algebra over $K$.  An {\em $R$-Hopf order in $B$}\index{Hopf order} is an $R$-Hopf algebra $H$ which satisfies the following additional conditions: (i) $H$ is a finitely generated, projective (hence free) $R$-submodule of $B$ and (ii) $K\otimes_R H\cong B$ as $K$-Hopf algebras.   Necessarily, the comultiplicaton, counit and coinverse maps of $H$ are induced from those of $B$.  If $H$ is an $R$-Hopf order in $B$, then its linear dual $H^*$ is an $R$-Hopf order in $B^*$.   

This paper concerns the construction of Hopf orders in $B$ in the case that $B$ is the 
group ring $K[C_p^n]$ or its linear dual $(K[C_p^n])^*$ where $C_p^n$, $n\ge 1$, denotes the elementary abelian group of order $p^n$ with 
$C_p^n=\langle g_1,g_2,\dots,g_n\rangle$, $g_i^p=1$, $1\le i\le n$. 

%Let $C_p^n$, $n\ge 1$, denote the elementary abelian group of order $p^n$ with 
%$C_p^n=\langle g_1,g_2,\dots,g_n\rangle$, $g_i^p=1$, $1\le i\le n$.  It is well-known that the group ring $K[C_p^n]$ and its linear dual $(K[C_p^n])^*$ are Hopf algebras over $K$.  

The problem of classifying $R$-Hopf orders in $K[C_p^n]$, $n\ge 1$, is difficult.  The classification is complete for the cases $n=1,2$ \cite{TO70}, \cite{EU17}, but the case $n\ge 3$ remains open.  On the dual side the problem has been solved by A. Koch \cite{Ko17}, who has given a complete characterization of all $R$-Hopf orders in $(K[C_p^n])^*$, $n\ge 1$.  From Koch's work we conclude that Hopf orders in $(K[C_p^n])^*$ are determined by $n(n+1)/2$ parameters.   

One might hope to use Koch's result to classify Hopf orders in $K[C_p^n]$ by first computing the Hopf orders in $(K[C_p^n])^*$ and then taking their duals to obtain all Hopf orders in $K[C_p^n]$   
%Indeed, a similar method has been used for the cases $n=1,2$ (see \cite[Theorem 2.3]{EU17} and \cite[Theorem 3.6]{EU17}).  
but difficulties arise in computing the duals in $K[C_p^n]$. 
At the very least we expect that Hopf orders in $K[C_p^n]$ can be classified using $n(n+1)/2$ parameters (since their duals in $(K[C_p^n])^*$ require exactly $n(n+1)/2$ parameters) but even this is not obvious.

In this paper we address the $n=3$ case and give a classification (under a mild condition) of $R$-Hopf orders in $K[C_p^3]$.  We achieve this by first computing all Hopf orders in $(K[C_p^3])^*$ using a cohomological argument employed by C. Greither in the $n=2$, characteristic $0$ case \cite[Part I]{Gr92}.    We  find that an $R$-Hopf
order in $(K[C_p^3])^*$ appears as
\[J= R[\pi^{i_1}(\xi_{1,0,0}-\mu\xi_{0,1,0}-\alpha\xi_{0,0,1}),\pi^{i_2}(\xi_{0,1,0}-\beta\xi_{0,0,1}),\pi^{i_3}\xi_{0,0,1}],\]
constructed using $6$ parameters, $i_1,i_2,i_3,\mu,\alpha,\beta$, which satisfy certain conditions.

Under a mild restriction we then compute their duals in $K[C_p^3]$, finding that the duals require $6$ parameters as well.   A dual Hopf order in $K[C_p^3]$ can be written as a ``truncated exponential" Hopf order 
\[H= R\left [{g_1-1\over\pi^{i_1}},{g_2g_1^{[\mu]}-1\over\pi^{i_2}},
{g_3g_1^{[\alpha]}(g_2g_1^{[\mu]})^{[\beta]}-1\over \pi^{i_3}}\right ]\]
using the same parameters, $i_1,i_2,i_3,\mu,\alpha,\beta$.  Hopf orders of this form have appeared in the work of 
N. Byott and G. Elder as realizable Hopf orders \cite[\S5.3]{BE18}. 

The author would like to thank Alan Koch for helpful discussions concerning \cite{Ko17}. 

\section{Koch's classification of Hopf orders in $(K[C_p^n])^*$}

Let $H$ be a Hopf algebra over a commutative ring with unity $D$.  An element $h\in H$ is {\em primitive} if $\Delta(h)=1\otimes h+h\otimes 1$.   The Hopf algebra $H$ is {\em primitively generated} if it is generated as a $D$-algebra by primitive elements.  The $K$-Hopf algebra $(K[C_p^n])^*$ is primitively generated. We have
\[(K[C_p^n])^* = K[t_1,t_2,\dots,t_n],\]
with $\Delta(t_i)=1\otimes t_i+t_i\otimes 1$ and $t_i^p=t_i$ for $1\le i\le n$, see \cite[Example 2.6]{Ko17}.

%\cite{EU17},
%\cite[\S 16]{TWE}.  

In \cite{Ko17} A. Koch has completely determined all $R$-Hopf orders in   
$(K[C_p^n])^*$ using the categorical equivalence between the category of primitively generated $R$-Hopf algebras and the category of finite 
$R[F]$-modules that are free as $R$-modules, i.e.,  the category of {\em Dieudonn\'{e} modules}\index{Dieudonn\'{e} modules} in the sense of 
\cite[\S 2]{Jo93}.  Here $F$ is an indeterminate in the non-commutative polynomial ring $R[F]$; $F$ satisfies the condition $Fa=a^pF$ for all $a\in R$.

We review Koch's result.   Let $J$ be a primitively generated $R$-Hopf algebra.  There exists
a matrix $A=(a_{i,j})\in \mathrm{Mat}_n(R)$ for which
\[J=R[u_1,u_2,\dots,u_n]/I\]
where $I$ is the ideal in $R[u_1,u_2,\dots,u_n]$ generated by $\{u_i^p-\sum_{j=1}^n a_{j,i}u_j\}_{1\le i\le n}$,
with $u_i$ primitive for $1\le i\le n$; $A$ is the matrix {\em associated to $J$}.    For a matrix $M=(m_{i,j})\in \mathrm{Mat}_n(K)$, let $M^{(p)}$ be the matrix whose
$i,j$th entry is $m_{i,j}^p$. 

\begin{theorem}[Koch] \label{koch} \hfill\break

(i)\ Let $n\ge 1$, let $\Theta=(\theta_{i,j})$ be a lower triangular matrix in $\mathrm{GL}_n(K)$ and let $A=\Theta^{-1}\Theta^{(p)}$.  Suppose that $A=(a_{i,j})\in \mathrm{Mat}_n(R)$.  Let 
\[J=R[u_1,u_2,\dots,u_n]/I\]
where $I$ is the ideal in $R[u_1,u_2,\dots,u_n]$ generated by $\{u_i^p-\sum_{j=1}^n a_{j,i}u_j\}_{1\le i\le n}$,
with $u_i$ primitive for $1\le i\le n$.  
%Then $J$ is a primitively generated $R$-Hopf algebra.  The matrix associated to $J$ is $A$.  
%Moreover, there is a Hopf embedding $J\rightarrow (K[C_p^n])^*$ 
There is a Hopf embedding $J\rightarrow (K[C_p^n])^*$ given as
\[u_i\mapsto \sum_{j=1}^n \theta_{j,i}t_j,\]
$t_j^p=t_j$, $t_j$ primitive.  The image of $J$ under this embedding is 
\[H_\Theta = R[\{\sum_{j=1}^n \theta_{j,i}t_j,1\le i\le n\}],\]
which is an $R$-Hopf order in $(K[C_p^n])^*$.  

\vspace{.2cm}

(ii)\ Let $J$ be an $R$-Hopf order in $(K[C_p^n])^*$.   Then $J$ is a primitively generated $R$-Hopf algebra
and there exists $A=(a_{i,j})\in \mathrm{Mat}_n(R)$ with 
\[J=R[u_1,u_2,\dots,u_n]/I\]
where $I$ is the ideal in $R[u_1,u_2,\dots,u_n]$ generated by $\{u_i^p-\sum_{j=1}^n a_{j,i}u_j\}_{1\le i\le n}$,
with $u_i$ primitive for $1\le i\le n$. 
Moreover, there exists a lower triangular matrix
$\Theta=(\theta_{i,j})\in \mathrm{GL}_n(K)$ so that 
\[J\cong H_\Theta=R[\{\sum_{j=1}^n \theta_{j,i}t_j,1\le i\le n\}]\subseteq (K[C_p^n])^*,\]
$t_j^p=t_j$, $t_j$ primitive.  Since $\Theta$ is lower triangular, $J$ is determined by $n(n+1)/2$ parameters.

%Let $A=(a_{i,j})=\Theta^{-1}\Theta^{(p)}$.  An an $R$-algebra, 
%\[J=R[u_1,u_2,\dots,u_n]/I\]
%where $I$ is the ideal in $R[u_1,u_2,\dots,u_n]$ generated by $\{u_i^p-\sum_{j=1}^n a_{j,i}u_j\}_{1\le i\le n}$,
%with $u_i$ primitive for $1\le i\le n$.  There is a Hopf isomorphism $J\rightarrow H_\Theta$ given as
%\[u_i\mapsto \sum_{j=1}^n \theta_{j,i}t_j.\]
\end{theorem}

\begin{proof} (Sketch)  

(i):\ Since $A=\Theta^{-1}\Theta^{(p)}$ is in $\mathrm{Mat}_n(R)$, \cite[Theorem 4.2]{Ko17} applies to show 
that $H_\Theta$ is an $R$-Hopf order in $(K[C_p^n])^*$.

For (ii):\ Since $(K[C_p^n])^*$ is primitively generated, $J$ is a primitively generated $R$-Hopf algebra.  
Let $A=(a_{i,j})\in \mathrm{Mat}_n(R)$ be the matrix associated 
to $J$, i.e.,
\[J=R[u_1,u_2,\dots,u_n]/I\]
where $I$ is the ideal in $R[u_1,u_2,\dots,u_n]$ generated by $\{u_i^p-\sum_{j=1}^n a_{j,i}u_j\}_{1\le i\le n}$,
with $u_i$ primitive for $1\le i\le n$.   Note that $J$ is generically isomorphic to $(R[C_p^n])^*$, i.e.,
\[K\otimes_R J\cong K\otimes_R (R[C_p^n])^*\cong (K[C_p^n])^*\] 
as $K$-Hopf algebras.   Thus 
by \cite[Corollary 4.2]{Ko17}, there exists a matrix 
$\Theta=(\theta_{i,j})\in \mathrm{GL}_n(K)$ with $\Theta A =\Theta^{(p)}$, hence 
\[A=\Theta^{-1}\Theta^{(p)}.\]

Let 
\[H_\Theta = R[\{\sum_{j=1}^n \theta_{j,i}t_j,1\le i\le n\}]\subseteq (K[C_p^n])^*,\]
$t_i^p=t_i$, $t_i$ primitive.  By \cite[Theorem 4.3]{Ko17}, 
$\Theta$ determines an isomorphism $J\rightarrow H_\Theta$ (an embedding of $J$ into $(K[C_p^n])^*$) 
defined by
\[u_i\mapsto \sum_{j=1}^n \theta_{j,i}t_j\]
for $1\le i\le n$.  

From \cite[Theorem 6.1]{Ko17}, we can assume that $\Theta$ is lower triangular thus $H_\Theta$ is determined by $n(n+1)/2$ parameters. 

\end{proof}

%Explicitly, $H_\Theta$ is the $R$-algebra generated by $\{\sum_{j=1}\theta_{j,i}t_j\}_{1\le i\le n}$. 
%As an $R$-Hopf algebra
%\[J = R[u_1,u_2,\dots,u_n]/I,\]
%where $I$ is the ideal generated by $\{u_i^p-\sum_{j=1}^n a_{j,i}u_j\}_{1\le i\le n}$, and where each $u_i$ is
%primitive.    

In the examples that follow we illustrate how to use Theorem \ref{koch}(i) to construct Hopf orders in the cases $n=1,2$.

\begin{example} \label{TO} $n=1$.  In this case we choose $\Theta = (\Theta_{1,1})\in \mathrm{GL}_1(K)\cong K^\times$,
with $\Theta_{1,1}=\pi^i$ for some integer $i$.  To construct an $R$-Hopf order in $(K[C_p])^*$, we require that
\[A=\Theta^{-1}\Theta^{(p)}=(\pi^{-i})(\pi^{pi})=(\pi^{(p-1)i})\]
is in $\mathrm{Mat}_1(R)$. Thus we require $i\ge 0$.  Let
\[J=R[u_1]/(u_1^p-\pi^{(p-1)i}u_1)\]
where $u_1$ is primitive.  Then there is a Hopf embedding $J\rightarrow (K[C_p])^*=K[t_1]$
defined as $u_1\mapsto \pi^it_1$.   The image of $J$ under this embedding is 
\[H_i = R[\pi^it_1]\subseteq (K[C_p])^*;\]
$H_i$ is a {\em Tate-Oort} Hopf order in $(K[C_p])^*$ \cite{TO70}, \cite[Theorem 2.3]{EU17}.

\end{example}

\begin{example} \label{k2} $n=2$.  In this case we choose $\Theta= \begin{pmatrix}
\pi^{i}& 0 \\
\theta & \pi^j \\
\end{pmatrix}\in \mathrm{GL}_2(K)$, where $i,j$ are integers.  To construct an $R$-Hopf order in 
$(K[C_p])^*$, we require that
\[A=\Theta^{-1}\Theta^{(p)}=\begin{pmatrix} 
\pi^{(p-1)i} & 0 \\
\pi^{-j}\theta^p-\pi^{(p-1)i-j}\theta & \pi^{(p-1)j} \\
\end{pmatrix}\in \mathrm{Mat}_2(R).\]
Thus we require $i,j\ge 0$ and 
\[\nu(\pi^{-j}\theta^p-\pi^{(p-1)i-j}\theta)\ge 0.\]
(The condition $A\in \mathrm{Mat}_2(R)$ is equivalent to the conditions $i,j\ge 0$, 
$\nu(\pi^{-j}\theta^p-\pi^{(p-1)i-j}\theta)\ge 0$.)

Let 
\[J=R[u_1,u_2]/I\]
with 
\[I=(u_1^p-\pi^{(p-1)i}u_1-(\pi^{-j}\theta^p-\pi^{(p-1)i-j}\theta)u_2, u_2^p-\pi^{(p-1)j}u_2),\]
$u_1,u_2$ primitive.  Then there is a Hopf embedding $J\rightarrow (K[C_p^2])^*=K[t_1,t_2]$ defined by 
\[u_1\mapsto \pi^it_1+\theta t_2,\quad u_2\mapsto \pi^jt_2.\]
The image of $J$ under this embedding is 
\[H_{i,j,\theta} = R[\pi^it_1+\theta t_2,\pi^jt_2];\]
$H_{i,j,\theta}$ is an $R$-Hopf order in $(K[C_p^2])^*$. 
\end{example}

\begin{remark}  Koch's classification applies more broadly to construct all of the $R$-Hopf orders in any primitively generated $K$-Hopf algebra, for instance $K[t]/(t^{p^n})$, which represents the group scheme $\alpha_{p^n}$.  
\end{remark}

\section{Classification of Hopf orders in $(K[C_p^2])^*$ using Cohomology}

In \cite{EU17} G. Elder and this author use Greither's cohomological method of \cite{Gr92} 
to classify Hopf orders in $(K[C_p^2])^*$ and hence 
recover case $n=2$ of Theorem \ref{koch}.   Let $C_p^2=\langle g_1,g_2\rangle$ and let $\xi_{1,0}$, $\xi_{0,1}$ be elements of $(K[C_p^2])^*$
defined by 
\[\langle \xi_{1,0}, (g_1-1)^j(g_2-1)^k\rangle = \delta_{1,j}\delta_{0,k},\]
\[\langle \xi_{0,1}, (g_1-1)^j(g_2-1)^k\rangle = \delta_{0,j}\delta_{1,k},\]
$0\le j,k\le p-1$, where $\langle - ,- \rangle: (K[C_p^2])^*\times K[C_p^2]\rightarrow K$ is the duality map.
Then $\xi_{1,0}$ and $\xi_{0,1}$ are primitive elements in $(K[C_p^2])^*$ with $\xi_{1,0}^p=\xi_{1,0}$, $\xi_{0,1}^p=\xi_{0,1}$.  The $K$-Hopf algebra $(K[C_p^2])^*$ is primitively generated by $\xi_{1,0}$, $\xi_{0,1}$,
\[(K[C_p^2])^*=K[\xi_{1,0},\xi_{0,1}].\]
(The $\xi_{0,1}$, $\xi_{0,1}$ play the role of $t_1,t_2$ in Koch's construction given above.)

%We note differences in notation in \cite{EU17} and \cite{Ko17}. 
%In \cite{EU17} the authors write $\xi_{1,0}$, $\xi_{0,1}$ for primitive generators in place of $t_1,t_2$ used in \cite{Ko17}.  For instance, a Tate-Oort Hopf order in 
%$(K[C_p])^*$ is written as $H_i^*=R[\pi^i\xi_{1,0}]$ in \cite{EU17}. 

For $x\in K$, let $\wp(x)=x^p-x$.

\begin{proposition} [Elder-U.]\label{elder-u}  \hfill\break

\vspace{.2cm}

(i)  Let $i_1,i_2\ge 0$ be integers, let $\mu$ be an element of $K$ satisfying $\nu(\wp(\mu))\ge i_2-pi_1$ and let
\[H_{i_1,i_2,\mu}^*=R[\pi^{i_1}(\xi_{1,0}-\mu\xi_{0,1}), \pi^{i_2}\xi_{0,1}].\]
Then 
$H_{i_1,i_2,\mu}^*$ is an $R$-Hopf order in $(K[C_p^2])^*$. 

\vspace{.2cm}

(ii)  Let $J$ be an $R$-Hopf order in $(K[C_p^2])^*$.  Then
\[J\cong H_{i_1,i_2,\mu}^*=R[\pi^{i_1}(\xi_{1,0}-\mu\xi_{0,1}), \pi^{i_2}\xi_{0,1}],\]
where $i_1,i_2\ge 0$ are integers and $\mu$ is an element of $K$ satisfying $\nu(\wp(\mu))\ge i_2-pi_1$. 
\end{proposition}

\begin{proof}  (Sketch)  

(i):\ We have
\begin{eqnarray*}
(\pi^{i_1}(\xi_{1,0}-\mu\xi_{0,1}))^p & = & \pi^{pi_1}(\xi_{1,0}-\mu^p\xi_{0,1}) \\
& = &  \pi^{(p-1)i_1}\pi^{i_1}(\xi_{1,0}-\mu\xi_{0,1}) + \pi^{pi_1}\mu\xi_{0,1}-\pi^{pi_1}\mu^p\xi_{0,1}\\ 
& = &  \pi^{(p-1)i_1}\pi^{i_1}(\xi_{1,0}-\mu\xi_{0,1}) -\wp(\mu)\pi^{pi_1-i_2}\pi^{i_2}\xi_{0,1}.\\ 
\end{eqnarray*}
Since $i_1\ge 0$ and $\nu(\wp(\mu))\ge i_2-pi_1$, 
$\pi^{i_1}(\xi_{1,0}-\mu\xi_{0,1})$ satisfies a monic polynomial with coefficients in $R[\pi^{i_2}\xi_{0,1}]$.  Moreover,

\[(\pi^{i_2}\xi_{0,1})^p = \pi^{(p-1)i_2}\pi^{i_2}\xi_{0,1},\]
and so, $i_2\ge 0$ guarantees that $\pi^{i_2}\xi_{0,1}$ satisfies a monic polynomial with coefficients in $R$. 
Thus the $R$-algebra $H^*_{i_1,i_2,\mu}$ is a free $R$-submodule of $(K[C_p^2])^*$.    

The $R$-algebra $H^*_{i_1,i_2,\mu}$ is an $R$-Hopf algebra with comultiplication map defined by 
\[\Delta(\pi^{i_1}(\xi_{1,0}-\mu\xi_{0,1})) = 
1\otimes (\pi^{i_1}(\xi_{1,0}-\mu\xi_{0,1}))+ (\pi^{i_1}(\xi_{1,0}-\mu\xi_{0,1}))\otimes 1,\]
\[\Delta(\pi^{i_2}\xi_{0,1}) = 1\otimes \pi^{i_2}\xi_{0,1}+\pi^{i_2}\xi_{0,1}\otimes 1,\]
counit map defined as
\[\varepsilon(\pi^{i_1}(\xi_{1,0}-\mu\xi_{0,1}))=0,\quad \varepsilon(\pi^{i_2}\xi_{0,1})=0,\]
and coinverse defined by
\[S(\pi^{i_1}(\xi_{1,0}-\mu\xi_{0,1})) = -\pi^{i_1}(\xi_{1,0}-\mu\xi_{0,1}),\quad S(\pi^{i_2}\xi_{0,1}) =-\pi^{i_2}\xi_{0,1}.\]
In addition, $K\otimes_R H^*_{i_1,i_2,\mu}\cong (K[C_p^2])^*$, as $K$-Hopf algebras.   Thus 
$H^*_{i_1,i_2,\mu}$  is an $R$-Hopf order in $(K[C_p^2])^*$. 

For (ii):\  The $R$-Hopf order $J$ induces a short exact sequence of $R$-Hopf algebras
\be \label{ses1}
R\rightarrow H\rightarrow J\rightarrow H'\rightarrow R
\ee
where $H'$, $H$ are Hopf orders in $(K[C_p])^*$.   By \cite[Theorem 2.3]{EU17}, $H'$ and $H$ are Tate-Oort Hopf orders in $(K[C_p])^*$, i.e., $H'=H_{i_1}^*\cong R[\pi^{i_1}\xi_{1,0}]$ and 
$H=H_{i_2}^*\cong R[\pi^{i_2}\xi_{0,1}]$ for integers $i_1,i_2\ge 0$. 

Let $\D_{i_1}^*=\Spec\ H_{i_1}^*$, $\D_{i_2}^*=\Spec\ H_{i_2}^*$, and 
$\D^*= \Spec\ J$ denote the corresponding group schemes over $R$.    From (\ref{ses1}) we obtain the short exact sequence of group  schemes
\be \label{ses2}
 0\longrightarrow \D_{i_1}^*\longrightarrow \D^*\longrightarrow \D_{i_2}^*\longrightarrow 0,
\ee
which over $K$ appears as
\[0\longrightarrow \Spec\ (K[C_p])^*\longrightarrow  \Spec\ (K[C_p^2])^*
\longrightarrow  \Spec\ (K[C_p])^*\longrightarrow 0.\]
The short exact sequence (\ref{ses2}) represents a class in 
$\mathrm{Ext}_{gt}^1(\D_{i_2}^*,\D_{i_1}^*)$,
which is the group of generically trivial extensions of $\D_{i_2}^*$ by $\D_{i_1}^*$.  
So computing 
$\mathrm{Ext}_{gt}^1(\D_{i_2}^*,\D_{i_1}^*)$, and ultimately the representing algebras of the middle terms of 
these short exact sequences, will completely determine $J$.  

Let ${\mathbb G}_a$ denote the additive group scheme represented by the $R$-Hopf algebra $R[x]$, 
$x$ indeterminate.  There 
is a faithfully flat resolution of group schemes

\[0\longrightarrow \D_{i_1}^*\longrightarrow {\mathbb G}_a\stackrel{\Psi}{\longrightarrow} {\mathbb G}_a\longrightarrow 0,\]
$\Psi(x)=x^p-\pi^{(p-1)i_1}x$, and through the long exact sequence we obtain an isomorphism 
\[\mathrm{coker}(\Psi: \Hom(\D_{i_2}^*,{\mathbb G}_a)^{\Circlearrowleft})_{\mathalpha{gt}}\cong 
\mathrm{Ext}_{gt}^1(\D_{i_2}^*,\D_{i_1}^*).\]
By \cite[Proposition 4.5]{EU17}, $\coker(\Psi:\mathrm{Hom}(\D_{i_2}^*,{\mathbb G})^{\Circlearrowleft})_{\mathalpha{gt}}$ is isomorphic to the additive subgroup of 
\[K/({\mathbb F}_p+\mathfrak{p}^{i_2-i_1})\] 
represented by those elements $\mu\in K$ 
satisfying $\nu(\wp(\mu))\ge i_2-pi_1$.   We conclude that classes $[\mu]$ in $K/({\mathbb F}_p+\mathfrak{p}^{i_2-i_1})$ 
correspond to equivalence classes of generically trivial short exact sequences
\[0\longrightarrow \D_{i_1}^*\longrightarrow \D^*\longrightarrow \D_{i_2}^*\longrightarrow 0.\]
Using a push-out diagram, the representing Hopf algebra of $\D^*$ is 
\[J\cong H_{i_1,i_2,\mu}^*=R[\pi^{i_1}(\xi_{1,0}-\mu\xi_{0,1}), \pi^{i_2}\xi_{0,1}].\]

\end{proof}

\begin{remark} \label{change} Note that $H_{i_1,i_2,\mu}^*$ is Koch's $H_{i,j,\theta}$ 
of Example \ref{k2} with 
$\mu = -\pi^{-i}\theta$,
$\xi_{1,0}=t_1$, $\xi_{0,1}=t_2$, $i_1=i$, and $i_2=j$.  With this change of notation 
\[H_{i_1,i_2,\mu}^*\cong R[u_1,u_2]/I\]
where $I$ is generated by 
\begin{eqnarray*}
u_1^p-\pi^{(p-1)i}u_1-(\pi^{-j}\theta^p-\pi^{(p-1)i-j}\theta)u_2 & =  &   
u_1^p-\pi^{(p-1)i_1}u_1+(\pi^{-i_2}\pi^{pi_1}\mu^p-\pi^{(p-1)i_1-i_2}\pi^{i_1}\mu)u_2 \\
& = &  u_1^p-\pi^{(p-1)i_1}u_1+\wp(\mu)\pi^{pi_1-i_2}u_2 \\
\end{eqnarray*}
and 
\[u_2^p-\pi^{(p-1)j}u_2=u_2^p-\pi^{(p-1)i_2}u_2.\]
Moreover, Koch's condition 
\[\nu(\pi^{-j}\theta^p-\pi^{(p-1)i-j}\theta)\ge 0\]
becomes
\[\nu(\wp(\mu))\ge i_2-pi_1.\]
\end{remark}

\section{Hopf orders in $(K[C_p^3])^*$}

We can extend Proposition \ref{elder-u} to the case $n=3$.  Let $C_p^3=\langle g_1,g_2,g_3\rangle$ and let $\xi_{1,0,0}$, $\xi_{0,1,0}$, $\xi_{0,0,1}$ be elements of $(K[C_p^3])^*$
defined by 
\[\langle \xi_{1,0,0}, (g_1-1)^j(g_2-1)^k(g_3-1)^l\rangle = \delta_{1,j}\delta_{0,k}\delta_{0,l},\]
\[\langle \xi_{0,1,0}, (g_1-1)^j(g_2-1)^k(g_3-1)^l\rangle = \delta_{0,j}\delta_{1,k}\delta_{0,l},\]
\[\langle \xi_{0,0,1}, (g_1-1)^j(g_2-1)^k(g_3-1)^l\rangle = \delta_{0,j}\delta_{0,k}\delta_{1,l},\]
$0\le j,k,l\le p-1$, where $\langle - ,- \rangle: (K[C_p^3])^*\times K[C_p^3]\rightarrow K$ is the duality map.  
Then $\xi_{1,0,0}$, $\xi_{0,1,0}$ and $\xi_{0,0,1}$ are primitive elements in $(K[C_p^3])^*$ with 
$\xi_{1,0,0}^p=\xi_{1,0,0}$, $\xi_{0,1,0}^p=\xi_{0,1,0}$ and $\xi_{0,0,1}^p=\xi_{0,0,1}$.  
The $K$-Hopf algebra $(K[C_p^3])^*$ is primitively generated by $\xi_{1,0,0}$, $\xi_{0,1,0}$, $\xi_{0,0,1}$; 
\[(K[C_p^3])^*=K[\xi_{1,0,0},\xi_{0,1,0},\xi_{0,0,1}].\]
%(The $\xi_{0,1}$, $\xi_{0,1}$ play the role of $t_1,t_2$ in Koch's construction \cite{Ko17}.)

\begin{proposition} \label{3case_i}  Let $i_1,i_2,i_3\ge 0$ be integers, let $\mu,\alpha,\beta$ be elements of $K$ satisfying $\nu(\wp(\mu))\ge i_2-pi_1$, $\nu(\wp(\alpha)+\wp(\mu)\beta)\ge i_3-pi_1$ and $\nu(\wp(\beta))\ge i_3-pi_2$.   Let
\[H^*_{i_1,i_2,i_3,\mu,\alpha,\beta} = R[\pi^{i_1}(\xi_{1,0,0}-\mu\xi_{0,1,0}-\alpha\xi_{0,0,1}),\pi^{i_2}(\xi_{0,1,0}-\beta\xi_{0,0,1}),\pi^{i_3}\xi_{0,0,1}].\]
Then $H^*_{i_1,i_2,i_3,\mu,\alpha,\beta}$ is an $R$-Hopf order in $(K[C_p^3])^*$. 
\end{proposition}

\begin{proof}  We have

\begin{eqnarray*}
(\pi^{i_1}(\xi_{1,0,0}-\mu\xi_{0,1,0}-\alpha\xi_{0,0,1}))^p & = &  \pi^{pi_1}(\xi_{1,0,0}-\mu^p\xi_{0,1,0}-
\alpha^p\xi_{0,0,1}) \\
& = &  \pi^{(p-1)i_1}\pi^{i_1}(\xi_{1,0,0}-\mu\xi_{0,1,0}-\alpha\xi_{0,0,1}) \\
&& +\ \pi^{pi_1}\mu\xi_{0,1,0}+\pi^{pi_1}\alpha\xi_{0,0,1}-\pi^{pi_1}\mu^p\xi_{0,1,0}-\pi^{pi_1}\alpha^p\xi_{0,0,1} \\
& = &  \pi^{(p-1)i_1}\pi^{i_1}(\xi_{1,0,0}-\mu\xi_{0,1,0}-\alpha\xi_{0,0,1}) \\
&& -\ \wp(\mu)\pi^{pi_1}\xi_{0,1,0}-\wp(\alpha)\pi^{pi_1}\xi_{0,0,1} \\
& = &  \pi^{(p-1)i_1}\pi^{i_1}(\xi_{1,0,0}-\mu\xi_{0,1,0}-\alpha\xi_{0,0,1}) \\
&& -\ \wp(\mu)\pi^{pi_1-i_2}\pi^{i_2}\xi_{0,1,0}+\wp(\mu)\beta\pi^{pi_1-i_2}\pi^{i_2}\xi_{0,0,1} \\
&&\quad -\ \wp(\mu)\beta\pi^{pi_1-i_3}\pi^{i_3}\xi_{0,0,1} -\wp(\alpha)\pi^{pi_1-i_3}\pi^{i_3}\xi_{0,0,1} \\
& = &  \pi^{(p-1)i_1}\pi^{i_1}(\xi_{1,0,0}-\mu\xi_{0,1,0}-\alpha\xi_{0,0,1}) \\
&& -\ \wp(\mu)\pi^{pi_1-i_2}\pi^{i_2}(\xi_{0,1,0}-\beta\xi_{0,0,1}) \\
&&\quad -\ (\wp(\alpha)+\wp(\mu)\beta)\pi^{pi_1-i_3}\pi^{i_3}\xi_{0,0,1}.
\end{eqnarray*}
Since $i_1\ge 0$, $\nu(\wp(\mu))\ge i_2-pi_1$ and $\nu(\wp(\alpha)+\wp(\mu)\beta)\ge i_3-pi_1$,
the generator \[\pi^{i_1}(\xi_{1,0,0}-\mu\xi_{0,1,0}-\alpha\xi_{0,0,1})\] 
satisfies a monic polynomial with coefficients in 
$R$-Hopf order $R[\pi^{i_2}(\xi_{0,1,0}-\beta\xi_{0,0,1}),\pi^{i_3}\xi_{0,0,1}]$. Thus the $R$-algebra
$H^*_{i_1,i_2,i_3,\mu,\alpha,\beta}$ is a free $R$-submodule of $(K[C_p^3])^*$.

The $R$-algebra $H^*_{i_1,i_2,i_3,\mu,\alpha,\beta}$ is an $R$-Hopf algebra with comultiplication map 
defined by 
\[\Delta(\pi^{i_1}(\xi_{1,0,0}-\mu\xi_{0,1,0}-\alpha\xi_{0,0,1})) = 
1\otimes (\pi^{i_1}(\xi_{1,0,0}-\mu\xi_{0,1,0}-\alpha\xi_{0,0,1}))+ (\pi^{i_1}(\xi_{1,0,0}-\mu\xi_{0,1,0}-\alpha\xi_{0,0,1}))\otimes 1,\]
\[\Delta(\pi^{i_2}(\xi_{0,1,0}-\beta\xi_{0,0,1})) = 
1\otimes (\pi^{i_2}(\xi_{0,1,0}-\beta\xi_{0,0,1}))+ (\pi^{i_2}(\xi_{0,1,0}-\beta\xi_{0,0,1}))\otimes 1,\]
\[\Delta(\pi^{i_3}\xi_{0,0,1}) = 1\otimes \pi^{i_3}\xi_{0,0,1}+\pi^{i_3}\xi_{0,0,1}\otimes 1,\]
counit map defined as
\[\varepsilon(\pi^{i_1}(\xi_{1,0,0}-\mu\xi_{0,1,0}-\alpha\xi_{0,0,1}))=0,\quad\varepsilon(\pi^{i_1}(\xi_{1,0}-\mu\xi_{0,1}))=0,\quad \varepsilon(\pi^{i_2}\xi_{0,1})=0,\]
and coinverse defined by
\[S(\pi^{i_1}(\xi_{1,0,0}-\mu\xi_{0,1,0}-\alpha\xi_{0,0,1})) = -(\pi^{i_1}(\xi_{1,0,0}-\mu\xi_{0,1,0}-\alpha\xi_{0,0,1})),\]
\[S(\pi^{i_1}(\xi_{1,0}-\mu\xi_{0,1})) = -\pi^{i_1}(\xi_{1,0}-\mu\xi_{0,1}),\quad S(\pi^{i_2}\xi_{0,1}) =-\pi^{i_2}\xi_{0,1}.\]
In addition, $K\otimes_R H^*_{i_1,i_2,i_3,\mu,\alpha,\beta}\cong (K[C_p^3])^*$, as $K$-Hopf algebras.   Thus 
$H^*_{i_1,i_2,i_3,\mu,\alpha,\beta}$  is an $R$-Hopf order in $(K[C_p^3])^*$. 
\end{proof}

We want to show that an arbitrary $R$-Hopf order $J$ in $(K[C_p^3])^*$ can be written as 
$H^*_{i_1,i_2,i_3,\mu,\alpha,\beta}$ for some parameters $i_1,i_2,i_3,\mu,\alpha,\beta$.   

An $R$-Hopf order $J$ induces a short exact sequence of $R$-Hopf algebras
\be \label{ses3}
R\rightarrow H\rightarrow J\rightarrow H'\rightarrow R
\ee
where $H'$ is an $R$-Hopf order in $(K[C_p^2])^*$ and $H$ is an $R$-Hopf order in $(K[C_p])^*$.  
By Proposition \ref{elder-u} (ii), $H'$ is of the form
\[H'\cong H_{i_1,i_2,\mu}^*= R[\pi^{i_1}(\xi_{1,0,0}-\mu\xi_{0,1,0}), \pi^{i_2}\xi_{0,1,0}],\]
for integers $i_1,i_2\ge 0$ and element $\mu\in K$ with $\nu(\wp(\mu))\ge i_2-pi_1$,
and by \cite[Theorem 2.3]{EU17}, $H$ is of the form
\[H\cong H_{i_3}^* = R[\pi^{i_3}\xi_{0,0,1}]\]
for integer $i_3\ge 0$. 

Let $\D^*=\Spec\ J$, $\D_{i_1,i_2,\mu}^*=\Spec\ H_{i_1,i_2,\mu}^*$ and 
$\D_{i_3}^*=\Spec\ H_{i_3}^*$.   Then from (\ref{ses3}) we obtain the short exact sequence of $R$-group schemes
\be \label{ses4}
0\longrightarrow \D_{i_1,i_2,\mu}^*\longrightarrow \D^*\longrightarrow \D_{i_3}^*\longrightarrow 0,
\ee
which over $K$ appears as
\[0\longrightarrow \Spec\ (K[C_p^2])^*\longrightarrow  \Spec\ (K[C_p^3])^*
\longrightarrow  \Spec\ (K[C_p])^*\longrightarrow 0.\]
The short exact sequence (\ref{ses4}) represents a class in 
$\mathrm{Ext}_{gt}^1(\D_{i_3}^*,\D_{i_1,i_2,\mu}^*)$,
which is the group of generically trivial extensions of $\D_{i_3}^*$ by $\D_{i_1,i_2,\mu}^*$.  
So computing 
$\mathrm{Ext}_{gt}^1(\D_{i_3}^*,\D_{i_1,i_2,\mu}^*)$, and ultimately the representing algebras of the middle terms of 
these short exact sequences, will completely determine $J$.  

As an $R$-Hopf algebra,

\[H_{i_1,i_2,\mu}^*\cong R[u_1,u_2]/I,\]
where
\[I=(u_1^p-\pi^{(p-1)i_1}u_1+\wp(\mu)\pi^{pi_1-i_2}u_2, u_2^p-\pi^{(p-1)i_2}u_2).\] 
(Recall the change in notation of Remark \ref{change}.)    Let ${\mathbb G}_a^2$ be the $R$-group scheme represented by the $R$-Hopf algebra $R[x,y]$, 
$x,y$ indeterminates.
There is a faithfully flat resolution
\be \label{res}
0\longrightarrow \D_{i_1,i_2,\mu}^*\longrightarrow {\mathbb G}_a^2\stackrel{\Psi}{\longrightarrow} {\mathbb G}_a^2\longrightarrow 0
\ee
where $\Psi = (\Psi_1(x,y),\Psi_2(x,y))$ with 
\[\Psi_1(x,y)= x^p-\pi^{(p-1)i_1}x+\wp(\mu)\pi^{pi_1-i_2}y\] 
and
\[\Psi_2(x,y) =y^p-\pi^{(p-1)i_2}y.\] 
From (\ref{res}) we obtain the long exact sequence

\[\Hom(\D_{i_3}^*,{\mathbb G}_a^{2})
\stackrel{\Psi}{\longrightarrow} \Hom(\D_{i_3}^*,{\mathbb G}_a^{2})\stackrel{\omega}{\longrightarrow} \mathrm{Ext}^1(\D_{i_3}^*,\D_{i_1,i_2,\mu}^*)\stackrel{\iota}{\longrightarrow} \mathrm{Ext}^1(\D_{i_3}^*,{\mathbb G}_a^{2}),\]
with connecting homomorphism $\omega$, which induces the map $\rho$ in the exact sequence

%\label{exact}
\[0\longrightarrow \mathrm{coker}(\Psi: \Hom(\D_{i_3}^*,{\mathbb G}_a^{2})^{\Circlearrowleft}) \stackrel{\rho}{\longrightarrow}\mathrm{Ext}^1(\D_{i_3}^*,\D_{i_1,i_2,\mu}^*)\stackrel{\iota}{\longrightarrow} \mathrm{Ext}^1(\D_{i_3}^*,{\mathbb G}_a^{2}).\]
Tensoring with $K$ and considering kernels (indicated by $gt$) results in the following commutative diagram with exact rows and columns:

\footnotesize
\[
\begin{array}{ccccccc}
&& 0 && 0 && 0  \\
&& \downarrow && \downarrow &&\downarrow \\
0&\longrightarrow &\mathrm{coker}(\Psi: \Hom(\D_{i_3}^*,{\mathbb G}_a^{2})^{\Circlearrowleft})_{gt} &\stackrel{\rho}{\longrightarrow}&\mathrm{Ext}^1_{gt}(\D_{i_3}^*,\D_{i_1,i_2,\mu}^*)&\stackrel{\iota}{\longrightarrow} &\mathrm{Ext}^1_{gt}(\D_{i_3}^*,{\mathbb G}_a^{2}) \\
&& \downarrow && \downarrow &&\downarrow \\
0&\longrightarrow &\mathrm{coker}(\Psi: \Hom(\D_{i_3}^*,{\mathbb G}_a^{2})^{\Circlearrowleft}) &\stackrel{\rho}{\longrightarrow}&\mathrm{Ext}^1(\D_{i_3}^*,\D_{i_1,i_2,\mu}^*)&\stackrel{\iota}{\longrightarrow} &\mathrm{Ext}^1(\D_{i_3}^*,{\mathbb G}_a^{2}) \\
&& \downarrow && \downarrow &&\downarrow \\
0&\longrightarrow &\mathrm{coker}(\Psi: \Hom(\D_{i_3,K}^*,{\mathbb G}_{a,K}^{2})^{\Circlearrowleft}) &\stackrel{\rho}{\longrightarrow}&\mathrm{Ext}^1(\D_{i_3,K}^*,\D_{i_1,i_2,\mu,K}^*)&\stackrel{\iota}{\longrightarrow} &\mathrm{Ext}^1(\D_{i_3,K}^*,{\mathbb G}_{a,K}^{2}) \\
\end{array}
\]
\normalsize
Now by \cite[Lemma 4.1]{HS71},
\[\mathrm{Ext}^1_{gt}(\D_{i_3}^*,{\mathbb G}_a^{2})=(\mathrm{Ext}^1_{gt}(\D_{i_3}^*,{\mathbb G}_a))^{2}.\]
Thus
\[\mathrm{Ext}^1_{gt}(\D_{i_3}^*,{\mathbb G}_a^{2})=0\]
since $\mathrm{Ext}^1_{gt}(\D_{i_3}^*,{\mathbb G}_a)=0$ by \cite[Proposition 4.3]{EU17}.  

We have shown the following.

\begin{proposition} \label{3-iso} There is an isomorphism 
\[\rho: \mathrm{coker}(\Psi: \Hom(\D_{i_3}^*,{\mathbb G}_a^{2})^{\Circlearrowleft})_{\mathalpha{gt}}\rightarrow  
\mathrm{Ext}_{gt}^1(\D_{i_3}^*,\D_{i_1,i_2,\mu}^*).\]
\end{proposition}

\begin{proof}   We have $\mathrm{Ext}^1_{gt}(\D_{i_3}^*,{\mathbb G}_a^2)=0$ in the first row of the commutative diagram above.

\end{proof}

The isomorphism of Proposition \ref{3-iso} classifies all short exact sequences
\[0\longrightarrow \D_{i_1,i_2,\mu}^*\longrightarrow \D^*\longrightarrow \D_{i_3}^*\longrightarrow 0,\]
which over $K$ appear as
\[0\longrightarrow \Spec\ (K[C_p^2])^*\longrightarrow  \Spec\ (K[C_p^3])^*
\longrightarrow  \Spec\ (K[C_p])^*\longrightarrow 0.\]
We seek to compute $J$, which is the representing algebra of $\D^*$.  To do this we first need to 
compute $\mathrm{coker}(\Psi: \Hom(\D_{i_3}^*,{\mathbb G}_a^{2})^{\Circlearrowleft})_{\mathalpha{gt}}$.  
Let 
\[\F_p(\mu,-1) = \{(\alpha,\beta)\in K^2\mid \alpha=m\mu, \beta = -m, m\in \F_p\}.\]

\begin{proposition} \label{coker} The group $\mathrm{coker}(\Psi: \Hom(\D_{i_3}^*,{\mathbb G}_a^{2})^{\Circlearrowleft})_{\mathalpha{gt}}$
is isomorphic to the additive subgroup of 
\[K^2/(\F_p(\mu,-1)+(\F_p+\mathfrak{p}^{i_3-i_1},\mathfrak{p}^{i_3-i_2}))\] 
represented by pairs $(\alpha,\beta)$
which satisfy $\nu(\wp(\alpha)+\wp(\mu)\beta)\ge i_3-pi_1$ and $\nu(\wp(\beta))\ge i_3-pi_2$. 
\end{proposition}

%\begin{remark}  We assume \cite[Convention 1]{EU17}, i.e., that $\beta$ is an element of largest valuation in
%the coset $\beta+\F_p+P^{i_3-i_2}$.  Thus, if $\beta\in \F_p+P^{i_3-i_2}$, then $\beta=0$. 
%\end{remark}

\begin{proof}  Let ${\mathcal P}=\Prim(H_{i_3}^*)=R\pi^{i_3}\xi_{0,0,1}$ denote the set of primitive elements of 
$H_{i_3}^*$.   Then elements of  $\mathrm{coker}(\Psi: \Hom(\D_{i_3}^*,{\mathbb G}_a^{2})^{\Circlearrowleft})_{\mathalpha{gt}}$ correspond to the quotient 
\[(\Psi(K\otimes_R {\mathcal P}^{2})\cap {\mathcal P}^{2})/\Psi({\mathcal P}^{2}).\]   
We give a description of this quotient.
Elements in $K\otimes_R {\mathcal P}^2$ can be written as $(\alpha\pi^{i_1}\xi_{0,0,1},\beta\pi^{i_2}\xi_{0,0,1})$
for some $\alpha,\beta\in K$.  Thus an element in $\Psi(K\otimes_R {\mathcal P}^{2})$ appears as 

\begin{eqnarray*}
\Psi(\alpha\pi^{i_1}\xi_{0,0,1},\beta\pi^{i_2}\xi_{0,0,1}) & = &  (\Psi_1(\alpha\pi^{i_1}\xi_{0,0,1},\beta\pi^{i_2}\xi_{0,0,1}),\Psi_2(\alpha\pi^{i_1}\xi_{0,0,1},\beta\pi^{i_2}\xi_{0,0,1})) \\
& = &  (\wp(\alpha)\pi^{pi_1}\xi_{0,0,1}+\wp(\mu)\beta\pi^{pi_1}\xi_{0,0,1},\wp(\beta)\pi^{pi_2}\xi_{0,0,1}) \\
& = &  ((\wp(\alpha)+\wp(\mu)\beta)\pi^{pi_1}\xi_{0,0,1},\wp(\beta)\pi^{pi_2}\xi_{0,0,1}).
\end{eqnarray*}
The image $\Psi(\alpha\pi^{i_1}\xi_{0,0,1},\beta\pi^{i_2}\xi_{0,0,1})$ is in ${\mathcal P}^{2}$ precisely when 
$\nu(\wp(\alpha)+\wp(\mu)\beta)\ge i_3-pi_1$ and $\nu(\wp(\beta))\ge i_3-pi_2$, and $\Psi(\alpha\pi^{i_1}\xi_{0,0,1},\beta\pi^{i_2}\xi_{0,0,1})$ is in $\Psi({\mathcal P}^2)$ if and only if $(\alpha,\beta)$ is in the subgroup
\[\F_p(\mu,-1)+(\F_p+\mathfrak{p}^{i_3-i_1},\mathfrak{p}^{i_3-i_2}).\] 

\end{proof}

\begin{proposition}  \label{p^3} The representing algebra of $\D^*$ can be written in the form
\[J\cong H^*_{i_1,i_2,i_3,\mu,\alpha,\beta} = R[\pi^{i_1}(\xi_{1,0,0}-\mu\xi_{0,1,0}-\alpha\xi_{0,0,1}),\pi^{i_2}(\xi_{0,1,0}-\beta\xi_{0,0,1}),\pi^{i_3}\xi_{0,0,1}],\]
where $\mu\in K$ represents a coset in $K/(\F_p+\mathfrak{p}^{i_2-i_1})$ with 
$\nu(\wp(\mu))\ge i_2-pi_1$ and the pair $(\alpha,\beta)\in K^2$ represents a coset in 
\[K^2/(\F_p(\mu,-1)+(\F_p+\mathfrak{p}^{i_3-i_1},\mathfrak{p}^{i_3-i_2}))\] 
with $\nu(\wp(\alpha)+\wp(\mu)\beta)\ge i_3-pi_1$, $\nu(\wp(\beta))\ge i_3-pi_2$. 
\end{proposition}

\begin{proof}   Let $[E]$ be a class in $\mathrm{Ext}_{gt}^1(\D_{i_3}^*,\D_{i_1,i_2,\mu}^*)$.   In view of the isomorphism of Proposition \ref{3-iso}, we have $[h]=\rho^{-1}([E])$ for some homomorphism 
$h: {\mathbb D}_{i_3}^*\rightarrow {\mathbb G}_a^2$.  Applying \cite[Remark 4.1]{EU17} to the resolution (\ref{res})
we conclude that the image of $h$ under the connecting homomorphism $\omega$ (inducing $\rho$) is the class 
$[{\mathbb D}_h^*]$ of the pull-back 
\[{\mathbb D}_h^*=\{(w,(x,y))\in {\mathbb D}_{i_3}^*\times {\mathbb G}_a^2:\ h(w)=\Psi(x,y)\}.\]

In view of Proposition \ref{coker}, the homomorphism $h: {\mathbb D}_{i_3}^*\rightarrow {\mathbb G}_a^2$ is determined by a Hopf algebra map $f: R[x,y]\rightarrow H_{i_3}^*$ given as
\[x\mapsto (\wp(\alpha)+\wp(\mu)\beta)\pi^{pi_1}\xi_{0,0,1},\quad y\mapsto \wp(\beta)\pi^{pi_2}\xi_{0,0,1}\]
for some $\alpha,\beta\in K$ with
$\nu(\wp(\alpha)+\wp(\mu)\beta)\ge i_3-pi_1$, $\nu(\wp(\beta))\ge i_3-pi_2$. 

The representing Hopf algebra ${\caH}_h^*$ of ${\mathbb D}_h^*$ arises from the push-out diagram

\[
\begin{array}{ccc}
{\caH}_h^*&\leftarrow&   R[x,y]\\
\uparrow&&\Psi\uparrow \\
{H}_{i_3}^*&\stackrel{f}{\leftarrow}&R[x,y]. \\
\end{array}
\]
Thus,
\begin{eqnarray*}
{\caH}_h^* & = & {R[\pi^{i_3}\xi_{0,0,1}]\otimes _R R[x,y]\over (\Psi(\alpha\pi^{i_1}\xi_{0,0,1},\beta\pi^{i_2}\xi_{0,0,1})\otimes 1+1\otimes \Psi(x,y))} \\
& \cong &  {R[\pi^{i_3}\xi_{0,0,1}][x,y]\over (\Psi(x,y)+\Psi(\alpha\pi^{i_1}\xi_{0,0,1},\beta\pi^{i_2}\xi_{0,0,1}))} \\
& \cong &  {R[\pi^{i_3}\xi_{0,0,1}][x,y]\over (\Psi(x+\alpha\pi^{i_1}\xi_{0,0,1},y+\beta\pi^{i_2}\xi_{0,0,1}))} \\
& \cong &  R[\pi^{i_1}(\xi_{1,0,0}-\mu\xi_{0,1,0}-\alpha\xi_{0,0,1}),\pi^{i_2}(\xi_{0,1,0}-\beta\xi_{0,0,1}),\pi^{i_3}\xi_{0,0,1}],
\end{eqnarray*}
since
\[H_{i_1,i_2,\mu}^*\cong R[x,y]/(\Psi_1(x,y),\Psi_2(x,y)).\]

\end{proof}

%\begin{corollary}  \label{p^3} An arbitrary $R$-Hopf order in $(K[C_p^3])^*$ can be written in the form
%\[H^*_{i_1,i_2,i_3,\mu,\alpha,\beta} = R[\pi^{i_1}(\xi_{1,0,0}-\mu\xi_{0,1,0}-\alpha\xi_{0,0,1}),\pi^{i_2}(\xi_{0,1,0}-\beta\xi_{0,0,1}),\pi^{i_3}\xi_{0,0,1}],\]
%where $\mu\in K$ represents a coset in $K/(\F_p+\mathfrak{p}^{i_2-i_1})$ with $\nu(\wp(\mu))\ge i_2-pi_1$, 
%and the pair $(\alpha,\beta)\in K^2$ represents a coset in 
%\[K^2/(\F_p(\mu,-1)+(\F_p+\mathfrak{p}^{i_3-i_1},\mathfrak{p}^{i_3-i_2}))\] 
%with $\nu(\wp(\alpha)+\wp(\mu)\beta)\ge i_3-pi_1$, $\nu(\wp(\beta))\ge i_3-pi_2$. 
%\end{corollary}

%\begin{proof} Clear.
%\end{proof}

\begin{remark}  \label{mu-cond} Let $H^*_{i_1,i_2,i_3,\mu,\alpha,\beta}$ be an $R$-Hopf order in $(K[C_p^3])^*$.
Assume \cite[Convention 1]{EU17}, i.e., if $\mu\in \F_p+\mathfrak{p}^{i_2-i_1}$, then $\mu=0$.  
Then it is natural to assume that 
\[\nu(\mu)\ge i_3-i_1\quad \hbox{and}\quad i_2\ge i_3.\]   
The motivation behind this assumption is as follows.  The pair $(\alpha,\beta)=(0,0)$ corresponds to the Hopf order
\[H^*_{i_1,i_2,i_3,\mu,0,0} = R[\pi^{i_1}(\xi_{1,0,0}-\mu\xi_{0,1,0}),\pi^{i_2}\xi_{0,1,0},\pi^{i_3}\xi_{0,0,1}]\]
and the pair $(\alpha,\beta)=(m\mu,-m)$, $m\in \F_p$, in $\F_p(\mu,-1)$ corresponds to the Hopf order

\begin{eqnarray*}
H^*_{i_1,i_2,i_3,\mu,m\mu,-m} & = & R[\pi^{i_1}(\xi_{1,0,0}-\mu\xi_{0,1,0}-m\mu\xi_{0,0,1}),\pi^{i_2}(\xi_{0,1,0}
+m\xi_{0,0,1}),\pi^{i_3}\xi_{0,0,1}] \\
& = &  R[\pi^{i_1}(\xi_{1,0,0}-\mu(\xi_{0,1,0}+m\xi_{0,0,1})),
\pi^{i_2}(\xi_{0,1,0}+m\xi_{0,0,1}),\pi^{i_3}\xi_{0,0,1}]. \\
\end{eqnarray*}
Since 
\[(\xi_{0,1,0}+m\xi_{0,0,1})^p=\xi_{0,1,0}+m\xi_{0,0,1},\]
and $\xi_{0,1,0}+m\xi_{0,0,1}$ is primitive, 
\[H^*_{i_1,i_2,i_3,\mu,0,0}\cong H^*_{i_1,i_2,i_3,\mu,m\mu,-m}\]
as $R$-Hopf algebras.  Thus the pair $(\alpha,\beta)=(\mu m,-m)$ corresponds to the trivial extension and so 
we assume that 

\[(\mu m,-m)\in (\F_p+\mathfrak{p}^{i_3-i_1},\mathfrak{p}^{i_3-i_2}),\]
Hence, $\nu(\mu)\ge i_3-i_1$ and $i_2\ge i_3$. 

\end{remark}

%Now, $\psi(\mu\pi^{i_1})=\wp(\mu)\pi^{pi_1}$, and so,
%\begin{multline*}
%{\caH}_h^*\cong R[\pi^{i_2}\xi_{0,1}][x]/(\psi(x)+\psi(\mu\pi^{i_1})\xi_{0,1}) =R[\pi^{i_2}\xi_{0,1}][x]/(\psi(x+\mu\pi^{i_1}\xi_{0,1})),  \\
%\end{multline*}
%(Note:\ since $\wp(\mu)\pi^{pi_1}\in R$, this says that $\mu\pi^{i_1}\in R$.)    So, with $x\mapsto \pi^{i_1}\xi_{1,0}$, under $R[x]\rightarrow R[x]/\psi(x)\cong R[\pi^{i_1}\xi_{1,0}]$, one obtains
%\[{\caH}_h^*\cong R[\pi^{i_1}(\xi_{1,0}-\mu\xi_{0,1}),\pi^{i_2}\xi_{0,1}].\]
%It follows that a class $[E]$ in $\Extgt({\mathbb D}_{i_2}^*,{\mathbb D}_{i_1}^*)$ can be represented by a short exact sequence of the form
%\[E_\mu:\ 0\rightarrow {\mathbb D}_{i_1}^*\longrightarrow \mathrm {Spec}\ R[\pi^{i_1}(\xi_{1,0}-\mu\xi_{0,1}),\pi^{i_2}\xi_{0,1}]\longrightarrow {\mathbb D}_{i_2}^*\rightarrow 0\]
%for some $\mu\in K$ with $\nu_K(\wp(\mu))\ge i_2-pi_1$.   This completes the proof of Theorem \ref{classif-grp-scheme-ext}.

\section{Hopf orders in $K[C_p^n]$, $n=1,2,3$}

In the previous sections we have classified Hopf orders $J$ in $(K[C_p^n])^*$ for the cases $n=1,2,3$. To construct 
Hopf orders in $K[C_p^n]$, $n=1,2,3$, we need to compute the linear duals $J^*$.   

\subsection{The case $n=1$}  Let $C_p=\langle g_1\rangle$, let $i\ge 0$ be an integer, and let 
\[E(i)= R\left [{g_1-1\over \pi^i}\right ].\] 
Then it is well-known that $E(i)$ is an 
$R$-Hopf order in $K[C_p]$ (see for instance \cite[\S 2]{EU17}).

\begin{proposition} [Elder-U.]    Let $i\ge 0$ be an integer and let $H_i^*=R[\pi^i\xi_1]$, $\xi_1$ primitive, $\xi_1^p=\xi_1$,
be an $R$-Hopf order in $(K[C_p])^*$.  Then
\[(H_i^*)^*=E(i).\]
\end{proposition}

\begin{proof}  See \cite[Proposition 2.2]{EU17}.
\end{proof}

So every $R$-Hopf order in $K[C_p]$ can be written in the form
\[E(i)=R\left [{g_1-1\over \pi^i}\right ]\] 
for some integer $i\ge 0$. 
\subsection{The case $n=2$}   For $x,y$ the {\em truncated exponential}\index{truncated exponential} is defined as 
\[x^{[y]} =\sum_{m=0}^{p-1} {y\choose m} (x-1)^m,\] 
where ${y\choose m}$ is the {\em generalized binomial coefficient}\index{generalized binomial coefficient}
\[{y\choose m} =y(y-1)(y-2)\cdots (y-m+1)/m!\]  

\begin{proposition} [Elder-U.] Let $C_p^2=\langle g_1,g_2\rangle$, let $i_1,i_2\ge 0$ be integers and let $\mu$ be an element of $K$ that satisfies $\nu(\wp(\mu))\ge i_2-pi_1$.  Let 
\[E(i_1,i_2,\mu) = R\left[\frac{g_1-1}{\pi^{i_1}},\frac{g_2g_1^{[\mu]}-1}{\pi^{i_2}}\right].\]
Then $E(i_1,i_2,\mu)$ is an $R$-Hopf order in $K[C_p^2]$. 
\end{proposition}

\begin{proof}  See \cite[Proposition 3.4]{EU17}. 
\end{proof}

\begin{proposition} [Elder-U.]  Let $i_1,i_2\ge 0$ be integers and let $\mu$ be an element of $K$ that satisfies $\nu(\wp(\mu))\ge i_2-pi_1$. Let
\[H_{i_1,i_2,\mu}^*=R[\pi^{i_1}(\xi_{1,0}-\mu\xi_{0,1}), \pi^{i_2}\xi_{0,1}]\]
be an $R$-Hopf order in $(K[C_p^2])^*$.  Then
\[(H_{i_1,i_2,\mu}^*)^*=E(i_1,i_2,\mu).\]
\end{proposition}

\begin{proof}  See \cite[Theorem 3.6]{EU17}.
\end{proof}

Thus an arbitrary $R$-Hopf order in $K[C_p^2]$ appears as
\[E(i_1,i_2,\mu) = R\left[\frac{g_1-1}{\pi^{i_1}},\frac{g_2g_1^{[\mu]}-1}{\pi^{i_2}}\right]\]
for integers $i_1,i_2\ge 0$ and $\mu\in K$ with $\nu(\wp(\mu))\ge i_2-pi_1$.

\subsection{The case $n=3$}  Let $K[C_p^3]=K\langle g_1,g_2,g_3\rangle$.  Let 
\[E(i_1,i_2,\mu) = R\left [{g_1-1\over\pi^{i_1}},{g_2g_1^{[\mu]}-1\over\pi^{i_2}}\right ]\]
be an $R$-Hopf order in $K[C_p^2]$.  Let $i_3\ge 0$, let $\alpha,\beta\in K$ and let

\[E(i_1,i_2,i_3,\alpha,\beta,\mu)=R\left [{g_1-1\over\pi^{i_1}},{g_2g_1^{[\mu]}-1\over\pi^{i_2}},
{g_3g_1^{[\alpha]}(g_2g_1^{[\mu]})^{[\beta]}-1\over \pi^{i_3}}\right ]\]
be a ``truncated exponential" algebra over $R$.  We want to find conditions so that $E(i_1,i_2,i_3,\alpha,\beta,\mu)$ is an $R$-Hopf order in $K[C_p^3]$.  We first prove some lemmas.

\begin{lemma}\label{unit-lem}  Assume the conditions $\nu(\wp(\alpha)+\wp(\mu)\beta)\ge i_3-pi_1$, $\nu(\wp(\beta))\ge i_3-pi_2$.  Then $g_1^{[\alpha]}(g_2g_1^{[\mu]})^{[\beta]}$ is a unit in $E(i_1,i_2,\mu)$. 
\end{lemma}

\begin{proof}  By \cite[(2)]{EU17},
\[(g_1^{[\alpha]}(g_2g_1^{[\mu]})^{[\beta]})(g_1^{[-\alpha]}(g_2g_1^{[-\mu]})^{[-\beta]}) = 1,\]
so it is a matter of checking that $g_1^{[\alpha]}(g_2g_1^{[\mu]})^{[\beta]}$ and $g_1^{[-\alpha]}(g_2g_1^{[-\mu]})^{[-\beta]}$ are in $E(i_1,i_2,\mu)$.   We show that $g_1^{[\alpha]}(g_2g_1^{[\mu]})^{[\beta]}\in E(i_1,i_2,\mu)$.  Now, 
\[(g_2g_1^{[\mu]})^{[\beta]}=\sum_{m=0}^{p-1} {\beta\choose m}\pi^{mi_2}
\left ({g_2g_1^{[\mu]}-1\over \pi^{i_2}}\right )^m.\]
We claim that ${\beta\choose m}\pi^{mi_2}\in R$ for $0\le m\le p-1$. 
If $\nu(\beta)\ge 0$, then clearly this holds.  So we assume that
$\nu(\beta)<0$, which yields $\nu({\beta\choose m})\ge -mi_2$, $0\le m\le p-1$.  Hence
$(g_2g_1^{[\mu]})^{[\beta]}\in E(i_1,i_2,\mu)$. 

We next show that $g_1^{[\alpha]}\in E(i_1,i_2,\mu)$.  We have
\[g_1^{[\alpha]} =\sum_{m=0}^{p-1} {\alpha\choose m}\pi^{mi_1}
\left ({g_1-1\over \pi^{i_1}}\right )^m.\]
We claim that ${\alpha\choose m}\pi^{mi_1}\in R$ for $0\le m\le p-1$.   If $\nu(\alpha)\ge 0$, then
clearly this holds.  So we assume that
$\nu(\alpha)<0$, so that $\nu(\wp(\alpha))=p\nu(\alpha)$.    If $\nu(\wp(\alpha))\not = \nu(\wp(\mu)\beta)$, then 
$p\nu(\alpha)\ge\nu(\wp(\alpha)+\wp(\mu)\beta)\ge i_3-pi_1$, thus $\nu(\alpha)\ge -i_1$, which yields
$\nu({\alpha\choose m})\ge -mi_1$, $0\le m\le p-1$.   So we assume that $\nu(\wp(\alpha)) = p\nu(\alpha)=\nu(\wp(\mu)\beta)$.  If $\nu(\beta)\ge 0$, then
\[p\nu(\alpha)=\nu(\wp(\mu)\beta) =  \nu(\wp(\mu))+\nu(\beta)\ge \nu(\wp(\mu))\ge i_2-pi_1,\]
thus $\nu({\alpha\choose m})\ge -mi_1$, $0\le m\le p-1$.   If $\nu(\beta)<0$, then $\nu(\beta)\ge i_3/p-i_2$ and so
\[p\nu(\alpha)=\nu(\wp(\mu))+\nu(\beta)\ge i_2-pi_1+i_3/p-i_2=i_3/p-pi_1,\]
thus $\nu({\alpha\choose m})\ge -mi_1$, $0\le m\le p-1$.

A similar argument shows that $g_1^{[-\alpha]}(g_2g_1^{[-\mu]})^{[-\beta]}\in E(i_1,i_2,\mu)$.

\end{proof}

We have the formula from \cite[Lemma 2.2]{BE05}:
\be\label{griff1}
\ee
\[(1+x+y+xy)^{[z]}=(1+x)^{[z]}(1+y)^{[z]}(1+\wp(z)Q(x,y))\in
{\mathbb F}_p[x,y,z]/(x^p,y^p),\] 
where 
\[Q(x,y)=((x+y+xy)^p-x^p-y^p-(xy)^p)/p\in (x,y)^p\subset{\mathbb Z}[x,y],\] 
$Q(x,y)^2\in (x^p,y^p)$. (See also \cite[pp. 9-10]{EU17}.)  We give a generalization of this formula. 

\begin{lemma}  \label{gen-griff2} 

\[((1+x+y+xy)^{[z]})^{[a]} = ((1+x)^{[z]}(1+y)^{[z]})^{[a]}(1+\wp(z)a Q(x,y))\]
in ${\mathbb F}_p[x,y,z,a]/(x^p,y^p)$. 

\end{lemma}

\begin{proof} From (\ref{griff1}) we obtain
\[((1+x+y+xy)^{[z]})^{[a]}=((1+x)^{[z]}(1+y)^{[z]}(1+\wp(z)Q(x,y)))^{[a]}\]
in ${\mathbb F}_p[x,y,z,a]/(x^p,y^p)$.   Put 
\[D=(1+x)^{[z]}(1+y)^{[z]},\]
so that
\[((1+x)^{[z]}(1+y)^{[z]}(1+\wp(z)Q(x,y)))^{[a]} = (D(1+\wp(z)Q(x,y)))^{[a]}.\] 
Now, by (\ref{griff1}),
\begin{eqnarray*}
(D(1+\wp(z)Q(x,y)))^{[a]} & =  &  ((1+(D-1))(1+\wp(z)Q(x,y)))^{[a]} \\
& = &  (1+(D-1))^{[a]}(1+\wp(z)Q(x,y))^{[a]}\\
&\quad & \cdot\ (1+\wp(a)Q(D-1,\wp(z)Q(x,y))). \\
\end{eqnarray*}
Now $Q(D-1,\wp(z)Q(x,y))=0$ since $x^p=y^p=0$ and $Q(x,y)^2=0$. Thus
\begin{eqnarray*}
(D(1+\wp(z)Q(x,y)))^{[a]} & = & (1+(D-1))^{[a]}(1+\wp(z)Q(x,y))^{[a]}\\
& = &  D^{[a]}(1+\wp(z)Q(x,y))^{[a]}.
\end{eqnarray*}
Moreover, $Q(x,y)^2=0$ implies that
\[(1+\wp(z)Q(x,y))^{[a]}=1+\wp(z)aQ(x,y).\]
Thus 
\[((1+x+y+xy)^{[z]})^{[a]}=((1+x)^{[z]}(1+y)^{[z]})^{[a]}(1+\wp(z)a Q(x,y))\]
in ${\mathbb F}_p[x,y,z,a]/(x^p,y^p)$. 

\end{proof}

\begin{proposition} \label{p^3dual}  Let $i_1,i_2,i_3\ge 0$.  Let $\mu,\alpha,\beta\in K$ with 
$\nu(\wp(\mu))\ge i_2-pi_1$, $\nu(\wp(\alpha)+\wp(\mu)\beta)\ge i_3-pi_1$, $\nu(\wp(\beta))\ge i_3-pi_2$, $\nu(\mu)\ge i_3-i_1$ and $i_2\ge i_3$. 
Then $E=E(i_1,i_2,i_3,\alpha,\beta,\mu)$ is an $R$-Hopf order in $K[C_p^3]$. 
\end{proposition}

\begin{proof}  By Lemma \ref{unit-lem}, $g_1^{[\alpha]}(g_2g_1^{[\mu]})^{[\beta]}$ is a unit in $H_{i_1,i_2,\mu}$.  Thus
$g_3\in E$ and so $E\otimes_R K=K[C_p^3]$.  Moreover,
\[\left  ({g_3g_1^{[\alpha]}(g_2g_1^{[\mu]})^{[\beta]}-1\over \pi^{i_3}}\right )^p=0\]
and so $E$ is an $R$-algebra which is a free $R$-submodule of $K[C_p^3]$.  

We next check that $E$ is a coalgebra, with counit and comultiplication induced by that of $K[C_p^3]$.  We have $\varepsilon(E)\subseteq R$.   In view of \cite[Proposition (31.2)]{TWE}, $E$ is a coalgebra if we can show that
\be \label{coalg}
\Delta(g_1^{[\alpha]}(g_2g_1^{[\mu]})^{[\beta]})\equiv g_1^{[\alpha]}(g_2g_1^{[\mu]})^{[\beta]}\otimes g_1^{[\alpha]}(g_2g_1^{[\mu]})^{[\beta]}
\ee
modulo $\pi^{i_3}E(i_1,i_2,\mu)\otimes E(i_1,i_2,\mu)$.

Let 
\[X=(g_1-1)/\pi^{i_1}\otimes 1,\quad  Y=1\otimes (g_1-1)/\pi^{i_1},\]
\[A= (g_2g_1^{[\mu]}-1)/\pi^{i_2}\otimes 1,\quad B= 1\otimes (g_2g_1^{[\mu]}-1)/\pi^{i_2}\]
\[T=(g_2-1)/\pi^{i_3}\otimes 1,\quad V= 1\otimes (g_2-1)/\pi^{i_3}.\]
By Lemma \ref{gen-griff2}

\[(((1+\pi^{i_1}X)(1+\pi^{i_1}Y))^{[\mu]})^{[\beta]}= ((1+\pi^{i_1}X)^{[\mu]}(1+\pi^{i_1}Y)^{[\mu]})^{[\beta]}(1+\wp(\mu)\beta Q(\pi^{i_1}X,\pi^{i_1}Y)).\]
By (\ref{griff1}),

\[((1+\pi^{i_1}X)(1+\pi^{i_1}Y))^{[\alpha]}=(1+\pi^{i_1}X)^{[\alpha]}(1+\pi^{i_1}Y)^{[\alpha]}(1+\wp(\alpha)Q(\pi^{i_1}X,\pi^{i_1}Y)).\]
Multiplication yields

\vspace{.2cm}

$((1+\pi^{i_1}X)(1+\pi^{i_1}Y))^{[\alpha]}(((1+\pi^{i_1}X)(1+\pi^{i_1}Y))^{[\mu]})^{[\beta]}$

\vspace{.2cm}

\hspace{3cm}$ = (1+\pi^{i_1}X)^{[\alpha]}(1+\pi^{i_1}Y)^{[\alpha]}((1+\pi^{i_1}X)^{[\mu]}(1+\pi^{i_1}Y)^{[\mu]})^{[\beta]}$

\vspace{.2cm}

\hspace{5cm}$\cdot\ (1+(\wp(\alpha)+\wp(\mu)\beta) Q(\pi^{i_1}X,\pi^{i_1}Y))$.

\vspace{.2cm}

\noindent Now the condition $\nu(\wp(\alpha)+\wp(\mu)\beta))\ge i_3-pi_1$ implies that

\[(\wp(\alpha)+\wp(\mu)\beta) Q(\pi^{i_1}X,\pi^{i_1}Y)\in \pi^{i_3}R[X,Y,A,B].\]
Thus

$((1+\pi^{i_1}X)(1+\pi^{i_1}Y))^{[\alpha]}(((1+\pi^{i_1}X)(1+\pi^{i_1}Y))^{[\mu]})^{[\beta]}$

\vspace{.2cm}

\hspace{3cm}$ \equiv (1+\pi^{i_1}X)^{[\alpha]}(1+\pi^{i_1}Y)^{[\alpha]}((1+\pi^{i_1}X)^{[\mu]}(1+\pi^{i_1}Y)^{[\mu]})^{[\beta]}$

\vspace{.2cm}

\noindent modulo $\pi^{i_3}R[X,Y,A,B]$. 

Moreover, $\nu(\mu)\ge i_3-i_1$ and $i_2\ge i_3$ imply that 
\[g_2-1\in \pi^{i_3}E(i_1,i_2,\mu),\]
hence
\[\Delta(g_2-1)=\pi^{i_3}T+\pi^{i_3}V+\pi^{2i_3}TV\in \pi^{i_3}R[X,Y,A,B].\]
Consequently, 

%$\nu(\wp(\alpha)+\wp(\mu)\beta))\ge i_3-pi_1$ yields

\vspace{.2cm}

$((1+\pi^{i_1}X)(1+\pi^{i_1}Y))^{[\alpha]}((1+\pi^{i_3}T)(1+\pi^{i_3}V)((1+\pi^{i_1}X)(1+\pi^{i_1}Y))^{[\mu]})^{[\beta]}$

\vspace{.2cm}

\hspace{1cm}$\equiv (1+\pi^{i_1}X)^{[\alpha]}(1+\pi^{i_1}Y)^{[\alpha]}((1+\pi^{i_3}T)(1+\pi^{i_3}V)(1+\pi^{i_1}X)^{[\mu]}(1+\pi^{i_1}Y)^{[\mu]})^{[\beta]}$

\vspace{.2cm}

\noindent modulo $\pi^{i_3}R\left [X,Y,A,B\right ]$. 

By formula (\ref{griff1}), the condition 
\[\nu(\wp(\beta))\ge i_3-pi_2\] implies 
\[((1+\pi^{i_2}A)(1+\pi^{i_2}B))^{[\beta]}\equiv (1+\pi^{i_2}A)^{[\beta]}(1+\pi^{i_2}B)^{[\beta]}\]
modulo $\pi^{i_3}R[X,Y,A,B]$.   

So if $\nu(\wp(\alpha)+\wp(\mu)\beta))\ge i_3-pi_1$, $\nu(\wp(\beta))\ge i_3-pi_2$ and $\nu(\mu)\ge i_3-i_1$ hold, then

\vspace{.2cm}

$((1+\pi^{i_1}X)(1+\pi^{i_1}Y))^{[\alpha]}((1+\pi^{i_3}T)(1+\pi^{i_3}V)((1+\pi^{i_1}X)(1+\pi^{i_1}Y))^{[\mu]})^{[\beta]}$

\vspace{.2cm}

$\quad \equiv (1+\pi^{i_1}X)^{[\alpha]}(1+\pi^{i_1}Y)^{[\alpha]}((1+\pi^{i_3}T)(1+\pi^{i_1}X)^{[\mu]})^{[\beta]}((1+\pi^{i_3}V)(1+\pi^{i_1}Y)^{[\mu]})^{[\beta]}$

\vspace{.2cm}

\noindent modulo $\pi^{i_3}R[X,Y,A,B]$, which is condition (\ref{coalg}). 

Finally, we show that $S(E)\subseteq E$.  But this follows from the coalgebra condition and observation that the coinverse map is 

\[m(I\otimes m)(I\otimes I\otimes m)\cdots (I^{2p-3}\otimes m)(I^{2p-3}\otimes \Delta)\cdots (I\otimes I\otimes \Delta)(I\otimes \Delta)\Delta,\] 
where $m: E\otimes E\rightarrow E$ denotes multiplication in $E$. 
\end{proof}

We have the following:  Suppose $i_1,i_2,i_3$ are non-negative integers.  Suppose  $\mu, \alpha, \beta\in K$ satisfy the conditions $\nu(\wp(\mu))\ge i_2-pi_1$, $\nu(\wp(\alpha)+\wp(\mu)\beta)\ge i_3-pi_1$, 
$\nu(\wp(\beta))\ge i_3-pi_2$, and $\nu(\mu)\ge i_3-i_1$.   Then by Proposition \ref{3case_i}
\[H^*_{i_1,i_2,i_3,\mu,\alpha,\beta} = R[\pi^{i_1}(\xi_{1,0,0}-\mu\xi_{0,1,0}-\alpha\xi_{0,0,1}),\pi^{i_2}(\xi_{0,1,0}-\beta\xi_{0,0,1}),\pi^{i_3}\xi_{0,0,1}]\]
is an $R$-Hopf order in $(K[C_p^3])^*$ and by Proposition \ref{p^3dual}
\[E(i_1,i_2,i_3,\mu,\alpha,\beta)=R\left [{g_1-1\over\pi^{i_1}},{g_2g_1^{[\mu]}-1\over\pi^{i_2}},
{g_3g_1^{[\alpha]}(g_2g_1^{[\mu]})^{[\beta]}-1\over \pi^{i_3}}\right ]\]
is an $R$-Hopf order in $K[C_p^3]$. 

We claim that these Hopf orders are duals to each other.

\begin{proposition} \label{3duals} With the conditions as above,
 \[(H^*_{i_1,i_2,i_3,\mu,\alpha,\beta})^*=E(i_1,i_2,i_3,\mu,\alpha,\beta).\]
\end{proposition}

\begin{proof}   To prove the result, we show that 
\[H^*_{i_1,i_2,i_3,\mu,\alpha,\beta}\subseteq  E(i_1,i_2,i_3,\mu,\alpha,\beta)^*\]
and 
\[\disc(H^*_{i_1,i_2,i_3,\mu,\alpha,\beta})=\disc(E(i_1,i_2,i_3,\mu,\alpha,\beta)^*).\]

For the containment:  $\{\rho_{a,b,c}\}$, $0\le a,b,c\le p-1$, where
\[\rho_{a,b,c} = (g_1-1)^a(g_2g_1^{[\mu]}-1)^b(g_3g_1^{[\alpha]}(g_2g_1^{[\mu]})^{[\beta]}-1)^c/\pi^{ai_1+bi_2+ci_3}\]
is an $R$-basis for $E(i_1,i_2,i_3,\mu,\alpha,\beta)$.  Let $x=g_1-1$, $y=g_2-1$, $z=g_3-1$.  Then
\[\pi^{ai_1+bi_2+ci_3}\rho_{a,b,c} \equiv x^a(\mu x+y)^b(\alpha+\beta\mu)x+\beta y+z)^c\quad \mmod (x,y,z)^2.\]
We have

\[\langle\xi_{1,0,0},x^a(\mu x+y)^b(\alpha+\beta\mu)x+\beta y+z)^c\rangle= \left\{
\begin{array}{ll}
1 & \mbox{if $a=1,b=0,c=0$} \\
\mu & \mbox{if $a=0,b=1,c=0$} \\
\alpha+\beta\mu & \mbox{if $a=0,b=0,c=1$} \\
0 & otherwise.
\end{array}
\right.
\]

\[\langle\xi_{0,1,0},x^a(\mu x+y)^b(\alpha+\beta\mu)x+\beta y+z)^c\rangle= \left\{
\begin{array}{ll}
0 & \mbox{if $a=1,b=0,c=0$} \\
1 & \mbox{if $a=0,b=1,c=0$} \\
\beta & \mbox{if $a=0,b=0,c=1$} \\
0 & otherwise.
\end{array}
\right.
\]

\[\langle\xi_{0,0,1},x^a(\mu x+y)^b(\alpha+\beta\mu)x+\beta y+z)^c\rangle= \left\{
\begin{array}{ll}
0 & \mbox{if $a=1,b=0,c=0$} \\
0 & \mbox{if $a=0,b=1,c=0$} \\
1 & \mbox{if $a=0,b=0,c=1$} \\
0 & otherwise.
\end{array}
\right.
\]
Thus

\[\langle\pi^{i_1}(\xi_{1,0,0}-\mu\xi_{0,1,0}-\alpha\xi_{0,0,1}),\rho_{1,0,0}\rangle = 1,\]
\[\langle\pi^{i_1}(\xi_{1,0,0}-\mu\xi_{0,1,0}-\alpha\xi_{0,0,1}),\rho_{0,1,0}\rangle = 0,\]
\[\langle\pi^{i_1}(\xi_{1,0,0}-\mu\xi_{0,1,0}-\alpha\xi_{0,0,1}),\rho_{0,0,1}\rangle = 0,\]
and so
\[\langle\pi^{i_1}(\xi_{1,0,0}-\mu\xi_{0,1,0}-\alpha\xi_{0,0,1}),\rho_{a,b,c}\rangle = \delta_{a,1}\delta_{b,0}\delta_{c,0}.\]
Moreover,
\[\langle\pi^{i_2}(\xi_{0,1,0}-\beta\xi_{0,0,1}),\rho_{a,b,c}\rangle = \delta_{a,0}\delta_{b,1}\delta_{c,0}\]
and
\[\langle\pi^{i_3}\xi_{0,0,1},\rho_{a,b,c}\rangle = \delta_{a,0}\delta_{b,0}\delta_{c,1}.\]
Thus
\[H_{i_1,i_2,i_3,\mu,\alpha,\beta}^*\subseteq  E(i_1,i_2,i_3,\mu,\alpha,\beta)^*.\]

Regarding the discriminant statement, there is a short exact sequence of $R$-Hopf orders

\[R\rightarrow H_{i_2,i_3,\beta}^*\rightarrow H_{i_1,i_2,i_3,\mu,\alpha,\beta}^*\rightarrow H_{i_1}^*\rightarrow R.\]
By \cite[Proposition 2.2]{EU17}, $\disc(H_{i_1}^*)=(\pi^{p(p-1)i_1})$, by \cite[Proposition 3.2]{EU17}, $\disc(H_{i_2,i_3,\beta}^*)=(\pi^{p^2(p-1)(i_2+i_3)})$ and by \cite[(22.17) Corollary]{TWE}
\begin{eqnarray*}
\disc(H_{i_1,i_2,i_3,\mu,\alpha,\beta}^*) & = &   \disc(H_{i_2,i_3,\beta}^*)^p\disc(H_{i_1}^*)^{p^2} \\
& = & (\pi^{p^3(p-1)(i_1+i_2+i_3)}).
\end{eqnarray*}

There is a short exact sequence of Hopf orders 

\[R\rightarrow E(i_1,i_2,\mu)\rightarrow E(i_1,i_2,i_3,\mu,\alpha,\beta)\rightarrow E(i_3)\rightarrow R,\]
which dualizes as
\[R\rightarrow H_{i_3}^*\rightarrow E(i_1,i_2,i_3,\mu,\alpha,\beta)^*\rightarrow H_{i_1,i_2,\mu}^*\rightarrow R.\]
By \cite[Proposition 2.2]{EU17}, $\disc(H_{i_3}^*)=(\pi^{p(p-1)i_3})$, by \cite[Proposition 3.2]{EU17}, 
$\disc(H_{i_1,i_2,\mu}^*)=(\pi^{p^2(p-1)(i_1+i_2)})$ and by \cite[(22.17) Corollary]{TWE}
\begin{eqnarray*}
\disc(E(i_1,i_2,i_3,\mu,\alpha,\beta)^*) & = &   \disc(H_{i_3}^*)^{p^2}\disc(H_{i_1,i_2,\mu}^*)^p \\
& = & (\pi^{p^3(p-1)(i_1+i_2+i_3)}).
\end{eqnarray*}
Thus 
\[\disc(H_{i_1,i_2,i_3,\mu,\alpha,\beta}^*) = \disc(E(i_1,i_2,i_3,\mu,\alpha,\beta)^*)\]
which completes the proof. 
\end{proof}

%\begin{example} Let 
%\[H=R\left [{\sigma_3-1\over\pi^{M_3}},{\sigma_2\sigma_3^{[-\mu_{2,3}]}-1\over\pi^{M_2}},{\sigma_1\sigma_3^{[-\mu_{1,3}]}(\sigma_2\sigma_3^{[-\mu_{2,3}]})^{[-\mu_{1,2}]}-1\over \pi^{M_1}}\right ]\]
%be the realizable truncated exponential Hopf order in $K[C_p^3]$ found in [BE18, \S 5.3] (characteristic $p$ case).  
%Its dual $H^*$ is an $R$-Hopf order in $(K[C_p^3])^*$, and hence is classified as above.  We have

%\[H^* = R[\pi^{M_3}(\xi_{1,0,0}+\mu_{2,3}\xi_{0,1,0}+\mu_{1,3}\xi_{0,0,1}),\pi^{M_2}(\xi_{0,1,0}+\mu_{1,2}%\xi_{0,0,1}),\pi^{M_1}\xi_{0,0,1}].\]
%\end{example}

We summarize as follows.  Let $H$ be an $R$-Hopf order in $K[C_p^3]$.  Then $H$ induces a short exact sequence of $R$-Hopf orders
\[R\rightarrow  E(i_1,i_2,\mu)\rightarrow H\rightarrow E(i_3)\rightarrow R\]
where $E(i_1,i_2,\mu)$ is an $R$-Hopf order in $K[C_p^2]$ and $E(i_3)$ is an $R$-Hopf order in $K[C_p]$. 
Necessarily $i_1,i_2,i_3\ge 0$ and $\nu(\wp(\mu))\ge i_2-pi_1$. 
As in Remark \ref{mu-cond} we assume that $\nu(\mu)\ge i_3-i_1$ and $i_2\ge i_3$. 

\begin{proposition}  With $H$ as above we have
\[H=R\left [{g_1-1\over\pi^{i_1}},{g_2g_1^{[\mu]}-1\over\pi^{i_2}},
{g_3g_1^{[\alpha]}(g_2g_1^{[\mu]})^{[\beta]}-1\over \pi^{i_3}}\right ]\]
where $\alpha,\beta$ are elements of $K$ satisfying
$\nu(\wp(\alpha)+\wp(\mu)\beta)\ge i_3-pi_1$ and $\nu(\wp(\beta))\ge i_3-pi_2$.
\end{proposition}

\begin{proof}  Given $H$, its dual $J=H^*$ is an $R$-Hopf order in $(K[C_p^3])^*$.  By Proposition \ref{p^3}, $J\cong H_{i_1,i_2,i_3,\mu,\alpha,\beta}^*$ where $\alpha,\beta\in K$ satisfy
$\nu(\wp(\alpha)+\wp(\mu)\beta)\ge i_3-pi_1$ and $\nu(\wp(\beta))\ge i_3-pi_2$.
By Proposition \ref{3duals}, $H=(H_{i_1,i_2,i_3,\mu,\alpha,\beta}^*)^* = E(i_1,i_2,i_3,\mu,\alpha,\beta)$. 

\end{proof}

\end{document}